\documentclass[11pt,a4paper]{article}
\usepackage{amsmath,amsthm,amssymb,amsfonts}
\usepackage{graphics,graphicx,enumerate}
\usepackage{eucal,latexsym}
\usepackage{color,cite}
\usepackage{comment,todonotes}
\newtheorem{thm}{Theorem}[section]

\newtheorem{lem}[thm]{Lemma}

\newtheorem*{ram}{Ramsey's Theorem}

\def\HH{{ \mathcal{H}}}

\bfseries\normalfont

\title{A Ramsey-type theorem for the matching number regarding connected graphs}

\author{
Ilkyoo Choi\thanks{
Ilkyoo Choi was supported by Basic Science Research Program through the National Research Foundation of Korea (NRF) funded by the Ministry of Education (NRF-2018R1D1A1B07043049), and also by Hankuk University of Foreign Studies Research Fund.
Department of Mathematics, Hankuk University of Foreign Studies, Yongin-si, Gyeonggi-do, Republic of Korea.
E-mail: \texttt{ilkyoo@hufs.ac.kr}
}
\and
Michitaka Furuya\thanks{Michitaka Furuya was supported by JSPS KAKENHI Grant number JP18K13449.
College of Liberal Arts and Sciences, Kitasato University, Sagamihara, Kanagawa 252-0373, Japan.
E-mail: \texttt{michitaka.furuya@gmail.com}}
\and
Ringi Kim\thanks{Ringi Kim was supported by the National Research Foundation of Korea (NRF) grant funded by the Korea government (MSIT)(NRF-2018R1C1B6003786)
Department of Mathematical Sciences, KAIST, Daejeon, Republic of Korea.
E-mail: \texttt{ringikim2@gmail.com}
}
\and
Boram Park\thanks{
Boram Park was supported by Basic Science Research Program through the National Research Foundation of Korea (NRF) funded by the Ministry of Science, ICT and Future Planning (NRF-2018R1C1B6003577).
Department of Mathematics, Ajou University, Suwon-si, Gyeonggi-do, Republic of Korea.
E-mail: \texttt{borampark@ajou.ac.kr}
}
}

\begin{document}

\maketitle

\begin{abstract}
A major line of research is discovering Ramsey-type theorems, which are results of the following form: given a graph parameter $\rho$, 
every graph $G$ with sufficiently large $\rho(G)$
contains a ``well-structured'' induced subgraph $H$ with large $\rho(H)$.
The classical Ramsey's theorem deals with the case when the graph parameter under consideration is the number of vertices; there is also a Ramsey-type theorem regarding connected graphs.

Given a graph $G$, the matching number and the induced matching number of $G$ is the maximum size of a matching and an induced matching, respectively, of $G$.
In this paper, we formulate Ramsey-type theorems for the matching number and the induced matching number regarding connected graphs.
Along the way, we obtain a Ramsey-type theorem for the independence number regarding connected graphs as well.
\end{abstract}

\section{Introduction}\label{sec1}
For a positive integer $n$, let $[n]$ denote the set $\{1,\ldots,n\}$.
All graphs considered in this paper are undirected, finite, and simple.
As usual, given a graph $G$, let $V(G)$ and $E(G)$ denote the vertex set and edge set, respectively, of $G$.
For an edge $e$, we will slightly abuse notation and let $V(e)$ denote the set of both ends of $e$.
We use $P_n$, $K_n$, and $K_{n,m}$ to denote a path on $n$ vertices, a complete graph on $n$ vertices, and a complete bipartite graph where each of the two parts has $n$ and $m$ vertices, respectively.

Ever since the classical Ramsey's theorem \cite{Ramsey} appeared in 1930, it has been a central topic in combinatorics and graph theory as it provides a fundamental link between discrete mathematics and other branches of mathematics.
Ramsey's theorem guarantees that a graph on sufficiently many vertices contains either a large complete graph or a large edgeless graph (as an induced subgraph).
However, when we restrict our attention to connected graphs, Ramsey's theorem is not quite satisfying since an edgeless graph is not connected.
Nonetheless, there exists a well-known Ramsey-type theorem for connected graphs, which we state as Theorem~\ref{lem2.0} (see~\cite{D});
it states that a connected graph on sufficiently many vertices contains a long path, a large complete graph, or a large star as an induced subgraph.

\begin{ram}[\cite{Ramsey}]
For a positive integer $n$, there exists an integer $N_0(n)$ such that
every graph on at least $N_0(n)$ vertices contains either $K_n$ or $\overline{K_n}$ as an induced subgraph.
\end{ram}

\begin{thm}[{{\cite[Proposition~9.4.1]{D}}}]
\label{lem2.0}
For a positive integer $n$, there exists an integer $N_1(n)$ such that every connected graph on at least $N_1(n)$ vertices contains $P_n$, $K_{n}$, or $K_{1,n}$ as an induced subgraph.
\end{thm}

It is natural to seek Ramsey-type theorems for other graph parameters.
Hence, a popular research problem in this area is to discover Ramsey-type theorems, which are statements of the following form:
given a graph parameter $\rho$, every graph $G$ with sufficiently large $\rho(G)$ contains a ``well-structured'' induced subgraph $H$ with large $\rho(H)$.
See~\cite{2016BoPiRaSz,F2018,matching,lozin2017, 2012ALR} for various recent results in this direction, some with applications to algorithmic results.

The \emph{matching number} and the \emph{induced matching number} of a graph $G$, denoted by $\alpha'(G)$ and $\alpha''(G)$, is the maximum size of a matching and an induced matching, respectively, of $G$.
We remark that Dabrowski, Demange, and Lozin~\cite{matching} (also in~\cite{lozin2017}) considered Ramsey-type results for the  matching number, yet,
their list of guaranteed 
graphs contains disconnected graphs since they do not restrict their attention to connected graphs.

This paper sheds lights on a Ramsey-type theorem for the induced matching number regarding connected graphs.
We use the aforementioned result to obtain a Ramsey-type theorem  for the matching number regarding connected graphs.
Along the way, we obtain a Ramsey-type theorem  for the independence number regarding connected graphs.

In order to state our main theorem, we name four classes of graphs.
See Figure~\ref{fig:beta_tilde}.
\begin{figure}[h!]
\begin{center}
\includegraphics[width=14.5cm]{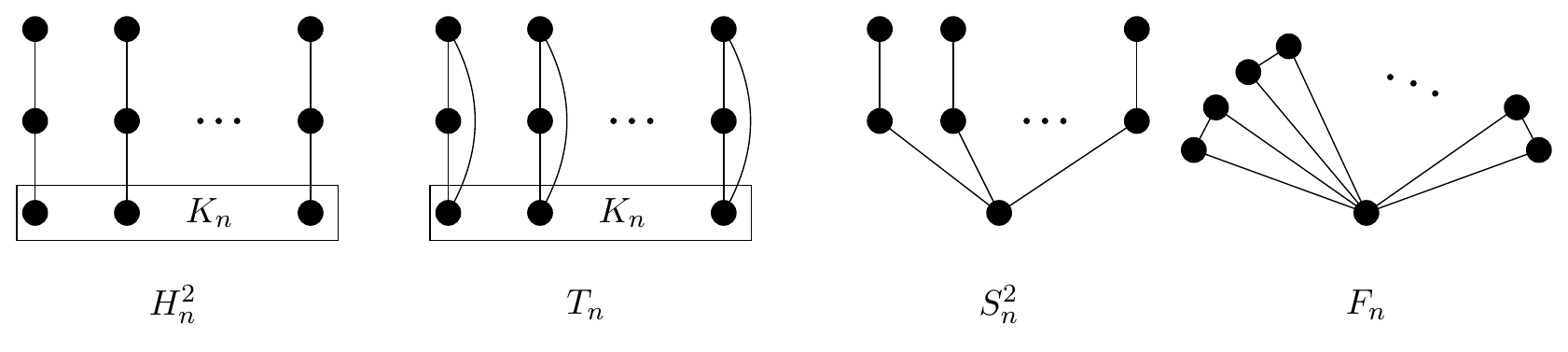}
\caption{The graphs $H_n^2$, $T_n$, $S_n^2$, and $F_n$.}
\label{fig:beta_tilde}
\end{center}
\end{figure}

\begin{itemize}
\item  $H_n^{l}$:
The graph obtained from the complete graph on $n$ vertices $v_1,\ldots,v_n$ by identifying $v_i$ and an end vertex of a path of length $l$ for each $i$.

\item $T_n$:
The graph obtained from the complete graph on $n$ vertices $v_1,\ldots,v_n$
by identifying $v_i$ and a vertex of a 3-cycle for each $i$.

\item $S_n^{l}$:
The graph obtained from $H_n^l$ by contracting the complete subgraph on $n$ vertices into a single vertex.

\item $F_n$:
The graph obtained from $T_n$ by contracting the complete subgraph on $n$ vertices into a single vertex.

\end{itemize}

We now state our main theorem, which is a Ramsey-type theorem  for the induced matching number regarding connected graphs.

\begin{thm}\label{THM:inducedmatching}
For a positive integer $n$, there exists an integer $N$ such that every connected graph $G$ with $\alpha''(G) \ge N$ contains $P_n$, $H_n^{2}$, $T_n$, $S_n^{2}$, or $F_n$ as an induced subgraph.
\end{thm}

Note that each of the guaranteed graphs is a connected graph with induced matching number at least $\frac{n-1}{3}$.
Using the above theorem, we also obtain a Ramsey-type theorem for the  matching number regarding connected graphs.

\begin{thm}
\label{THM:matching}
For a positive integer $n$, there exists an integer $N$ such that every connected graph $G$ with $\alpha'(G) \ge N$ contains $P_n$, $K_n$, $K_{n,n}$, $S_n^{2}$, or $F_n$ as an induced subgraph.
\end{thm}

Note that each of the guaranteed graphs is a connected graph with matching number at least $\frac{n-1}{2}$.

Before ending this section, we mention that Theorem~\ref{THM:matching} implies Ramsey-type theorems  for some other graph parameters regarding connected graphs.
Given a graph $G$, the \textit{vertex cover number}, denoted $\beta(G)$, is the minimum size of a subset $S$ of $V(G)$ such that $S\cap V(e)=\emptyset$ for all $e\in E(G)$.
It is not hard to see that every graph $G$ satisfies $\alpha'(G) \le \beta(G) \le 2\alpha'(G)$.
Also, for a graph $G$, the \textit{fractional matching number}, denoted $\alpha'_f(G)$, is the maximum of $\sum_{e\in E(G)}\varphi(e)$ over all functions $\varphi:E(G)\rightarrow [0,1]$ satisfying $\sum_{uv\in E(G)} \varphi(uv) \le 1$ for each vertex $v$.
It is known that $\alpha'(G)\le \alpha'_f(G)\le \frac{3}{2}\alpha'(G)$ (see \cite{CKO}).
Hence, by replacing $\alpha'(G)$ with $\beta(G)$ and $\alpha'_f(G)$ in Theorem~\ref{THM:matching}, we again have a Ramsey-type theorem for the vertex cover number and the fractional matching number, respectively, regarding connected graphs.

The paper is organized as follows. In Section~\ref{sec2}, we first prove a Ramsey-type theorem for the independence number regarding connected graphs, which will be used in our proof of Theorem~\ref{THM:inducedmatching}.
We prove Theorem~\ref{THM:matching} in Section~\ref{sec3}.
In Section~\ref{SEC:inducedpacking}, we discuss a generalization of our main result to the induced $\HH$-matching number.


\section{
A Ramsey-type theorem for the induced matching number regarding connected graph
}\label{sec2}
Our goal in this section is to prove Theorem~\ref{THM:inducedmatching}.
We start with a lemma that will come in handy later.

\begin{lem}
\label{LEM:increasing}
Let $G$ be a connected graph and let $T$ be a connected subgraph of $G$.
If $v_1,\ldots,v_m \in V(T)$ are cut-vertices of $G$, but are not cut-vertices of $T$,
then there exists an independent set $\{v_1',\ldots,v_m'\} \subseteq V(G)\setminus V(T)$ such that $N_{G}(v_i') \cap V(T) =\{v_i\}$ for all $i \in [m]$.
\end{lem}
\begin{proof}
For each $i\in[m]$,
since $v_{i}$ is a cut-vertex of $G$ but not a cut-vertex of $T$,
there exists a component $D_{i}$ of $G-v_{i}$ that does not contain vertices of $T$.
Note that $D_i$ is a component of $G\setminus V(T)$ where $v_i$ is the only neighbor of $V(D_i)$ in $V(T)$. Therefore, $D_1,\ldots,D_m$ are distinct components of $G\setminus V(T)$.

Let $v'_{i}\in N_{G}(v_{i})\cap V(D_{i})$. 
Then, we know $N_{G}(v'_{i})\cap V(T)=\{v_{i}\}$ for all $i\in[m]$.
In particular, since $V(D_1),\ldots,V(D_m)$ are pairwise disjoint, we conclude that $\{v_1',\ldots,v_m'\}$ is the desired independent set.
\end{proof}

We now prove a Ramsey-type theorem (Theorem~\ref{THM:independence}) for the independence number regarding connected graphs.
Note that all graphs that are guaranteed in Theorem~\ref{THM:independence} have independence number at least $n\over 2$.
This theorem will be used in the proof of  Theorem~\ref{THM:inducedmatching}.
We remark that we are not interested in optimizing constants, so the existence of $N_1$ in Theorem~\ref{lem2.0} is sufficient for our purposes.

\begin{thm}
\label{THM:independence}
For a positive integer $n$, every connected graph $G$ with independence number at least $N_1(n+1)$ contains $P_n$, $H_n^1$, or $K_{1,n}$ as an induced subgraph.
\end{thm}

\begin{proof}
Let $U$ be an independent set of $G$ where $|U|=N_1(n+1)$.
We may assume that every vertex of $G \setminus U$ is a cut-vertex of $G$ by deleting non cut-vertices one by one in $V(G)\setminus U$.

Since $|V(G)|\geq |U|= N_1(n+1)$, it follows from Theorem~\ref{lem2.0} that $G$ contains $P_{n+1}$, $K_{n+1}$, or $K_{1,n+1}$ as an induced subgraph.
Since we are done if $G$ contains either $P_{n+1}$ or $K_{1,n+1}$, we may assume that $G$ contains $K_{n+1}$.

Since $U$ is an independent set, every complete subgraph of $G$ contains at most one vertex of $U$.
So, $G$ contains an induced subgraph $K$ isomorphic to $K_n$ that does not contain a vertex in $U$.
Let $v_1,\ldots,v_{n}$ be the vertices of $K$.
Note that $v_i$ is a cut-vertex of $G$ since $v_i\not\in U$.
However, since $K$ is a complete graph, $v_i$ is not a cut-vertex of $K$.
Hence, by Lemma~\ref{LEM:increasing}, we can extend $K$ to an induced subgraph of $G$ that is isomorphic to $H_n^1$.
This completes the proof.
\end{proof}

Now we are ready to prove Theorem~\ref{THM:inducedmatching}.

\begin{proof}[Proof of Theorem~\ref{THM:inducedmatching}]
We show that the theorem holds when $N=N_1(3n+2)$.

Let $M=\{e_1,\ldots,e_N\}$ be an induced matching of $G$.
We may assume that every vertex of $G$ that is not incident with an edge in $M$ is a cut-vertex by deleting non-cut-vertices one by one from 
$V(G)\setminus \left(\bigcup_{e_i\in M}V(e_i)\right)$.

Let $G'$ be the graph obtained from $G$ by contracting $e_i$ into a single vertex $u_{i}$ for each $i\in[N]$.
Since $M$ is an induced matching of $G$, the set $U:=\{u_{1}, \ldots ,u_{N}\}$ is an independent set of $G^{'}$.
Note that every vertex in $V(G')\setminus U$ is a cut-vertex of $G'$.
Now, by Theorem~\ref{THM:independence}, $G'$ contains  either $P_{3n+1}$, $H_{3n+1}^1$, or $K_{1,3n+1}$ as an induced subgraph.
\\
\\
{\bf Case 1:}  $G'$ contains $P_{3n+1}$ as an induced subgraph.

Let $P'$ be an induced path of length $3n$ in $G'$. 
We will show that $G$ contains an induced path of length at least $3n$.
For each $u_i\in V(P')\cap U$, expand $u_i$ back into $e_i$ in $G$.
If $e_i$ has an end $z_i$ adjacent to both neighbors of $u_i$ on $P'$ in $G'$, then we keep the end $z_i$ and delete the other end.
Note that no neighbor of $u_i$ on $P'$ in $G'$ is in $U$ since $U$ is an independent set.
If $e_i$ has no such ends, then we keep both ends.
It is not hard to see that the graph obtained above is an induced path, and its length is at least $3n$.
Therefore, $G$ contains $P_n$ as an induced subgraph.
\\
\\
{\bf Case 2:} $G'$ contains $H_{3n+1}^1$ as an induced subgraph.

Since $U$ is an independent set, $G'$ contains an induced subgraph $H$ isomorphic to $H_{3n}^1$ whose the maximum clique of $H$ does not intersect $U$;
in other words, $H$ consists of a clique $C$ and an independent set $I$ where $|C|=|I|=3n$, $C$ does not intersect $U$, and the edges between $C$ and $I$ form a (perfect) matching.

Suppose $I\cap U$ contains $2n$ vertices, and without loss of generality, assume $u_1,\ldots,u_{2n}\in I\cap U$.
Let $v_i$ be the vertex of $C$ adjacent to $u_i$ for $i\in [2n]$, and let $H'$ be the induced subgraph of $G'$ on $\{u_i: i\in[2n]\}\cup\{v_i: i\in[2n]\}$.
We consider the induced subgraph of $G$ obtained from $H'$ by expanding each vertex $u_i$ back into the edge $e_i$ for $i\in[2n]$.
By the Pigeonhole principle, in $\{e_1,\ldots,e_{2n}\}$ there exist $n$ edges $e_{i_1},\ldots,e_{i_n}$ such that for every $j\in [n]$ either both ends of $e_{i_j}$ are adjacent to $v_{i_j}$ or only one end of $e_{i_j}$ is adjacent to $v_{i_j}$.
The graph induced by 
$\bigcup_{j\in[n]} \left(\{v_{i_j}\}\cup V(e_{i_j})\right)$
is isomorphic to $T_n$ and $H_n^2$ in the former and latter, respectively.

So, we may assume that $I\setminus U$ contains $n$ vertices, and without loss of generality assume $v_1',\ldots,v_n'\in I\setminus U$.
Let $v_i$ be the vertex in $C$ adjacent to $v_i'$.
Let $H^*$ be the induced subgraph of $G'$ on $\{v_i: i\in [n]\}\cup\{v_i': i\in [n]\}$.
Since $v_i'\not\in U$, it is a cut-vertex of $G'$ but not a cut-vertex of $H^*$.
By Lemma~\ref{LEM:increasing}, there exists an independent set $I'$ of size $n$ in $G'\setminus V(H^*)$ 
such that the induced subgraph $H^{**}$ of $G'$ on $V(H^*)\cup I'$ is isomorphic to $H_n^2$. 
Now, consider the induced subgraph of $G$ obtained from $H^{**}$ by expanding every vertex in $I'\cap U$ back to its original edge in $G$.
This graph contains $H_n^2$ as an induced subgraph.
Therefore, $G$ contains $H_n^2$ as an induced subgraph.
\\
\\
{\bf Case 3:} $G'$ contains $K_{1,3n+1}$ as an induced subgraph.

Let $S$ be an induced subgraph of $G'$ isomorphic to $K_{1,3n+1}$.
Let $c$ be the center of $S$ and let $L$ be the set of leaves of $S$.

Suppose $c \in U$, so that every leaf of $S$ is not in $U$ since $U$ is an independent set.
Let $e$ be the edge of $G$ that was contracted to obtain $c$.
There exists an end $x$ of $e$ and a subset $L'\subseteq L$ of size $n$ such that $x$ is adjacent (in $G$) to every vertex in $L'$.
Let $L'=\{v_1, \ldots, v_n\}$, and let $S'$ be the induced subgraph of $G'$ on $\{c\} \cup L'$.
Since $v_i$ is a cut-vertex of $G'$ but not a cut-vertex of $S'$ for $i\in[n]$, it follows from Lemma~\ref{LEM:increasing} that there exists an independent set $I\subseteq V(G')\setminus V(S')$ of size $n$ such that $\{c\} \cup L' \cup I$ is an induced subgraph $S^*$ of $G'$ that is isomorphic to $S_n^2$.
Now, we consider the induced subgraph of $G$ obtained from $S^*$ by expanding every vertex in $V(S^*)\cap U$ back to its original edge.
This graph contains $S_n^2$ as an induced subgraph.
Therefore, $G$ contains $S_n^2$ as an induced subgraph.

So, we may assume $c\not\in U$.
If there are $n$ leaves of $S$ not in $U$, then, by the same argument as above, $G$ contains an induced subgraph isomorphic to $S_n^2$.
So, we may further assume that there exist $2n$ leaves of $S$ contained in $U$.
Without loss of generality, let $u_1,\ldots,u_{2n}$ be such vertices.
By the Pigeonhole principle, in $e_1,\ldots,e_{2n}$, there exist $n$ edges $e_{i_1},\ldots,e_{i_n}$ such that, for every $j\in [n]$, either both ends of $e_{i_j}$ are adjacent to $c$ or only one end of $e_{i_j}$ is adjacent to $c$.
Now, the graph induced by $\{c\}\cup\bigcup_{j\in [n]} V(e_{i_j})$ is isomorphic to $F_n$ and $S_n^2$ in the former and latter, respectively.

Therefore, $G$ contains $P_n$, $H_n^{2}$, $T_n$, $S_n^{2}$, or $F_n$ as an induced subgraph.
\end{proof}

\section{
A Ramsey-type theorem for the matching number regarding connected graphs
}
\label{sec3}

In this section, we prove Theorem~\ref{THM:matching} using Theorem~\ref{THM:inducedmatching}.
Given a positive integer $n$, let $R_k(n)$ be the minimum integer $N$ such that every $k$-edge-coloring of a complete graph on $N$ vertices has a monochromatic complete graph on $n$ vertices.
See \cite{Ramsey}.

\begin{proof}[Proof of Theorem~\ref{THM:matching}.]
Given an integer $n$, let $r$ be an integer satisfying the following: if a connected graph has an induced matching of size $r$, then it also contains  $P_n$, $H_n^{2}$, $T_n$, $S_n^{2}$, or $F_n$ as an induced subgraph; such an $r$ is guaranteed to exist by Theorem~\ref{THM:inducedmatching}.
Since $K_n$ is an induced subgraph of $H_n^{2}$ and $T_n$, it is sufficient to show that a graph with sufficiently large matching number contains
$K_n$, $K_{n, n}$, or an induced matching of size $r$. 

Suppose that a connected graph $G$ has a matching $M=\{x_1y_1,\ldots ,x_my_m\}$ where $m\ge R_{16}(2r)$ and $r\geq n$.
Let $H$ be a complete graph on $m$ vertices where each vertex $v_k$ of $H$ corresponds to the edge $x_ky_k\in M$, and for $i< j$, color each edge $v_iv_j$ of $H$ with the color $(a,b,c,d)$ where $a,b,c,d \in \{0,1\}$ in the following fashion:
\begin{itemize}
\item $a=1$ if and only if $x_ix_j$ is an edge of $G$,
\item $b=1$ if and only if $y_iy_j$ is an edge of $G$,
\item $c=1$ if and only if $x_iy_j$ is an edge of $G$,
\item $d=1$ if and only if $y_ix_j$ is an edge of $G$.
\end{itemize}
Now, $H$ is a 16-edge-colored complete graph on at least $R_{16}(2r)$ vertices, and by Ramsey's Theorem  there is a monochromatic complete graph $H'$ on $2r$ vertices.
We may assume that $V(H')=\{v_1,\ldots,v_{2r}\}$, and the edges of $H'$ are all colored with $(a,b,c,d)$.

If $a=1$, then for $i, j\in [2r]$, we know $x_ix_j$ exists in $G$,
so $G$ has a complete graph $K_{2r}$ as an induced subgraph.
Since the argument is symmetric for $b=1$, we may assume that $a=b=0$.

If $(a,b,c,d)=(0,0,0,0)$, then for $i, j\in[2r]$, we know neither $x_iy_j$ nor $x_jy_i$ exists in $G$, so the induced subgraph of $G$ on $\{x_i: i\in[r]\}\cup\{y_i: i\in[r]\}$ has an induced matching of size $r$.

If $(a,b,c,d)=(0,0,1,0)$, then for $i, j\in[2r]$ with $i<j$, we know all edges $x_iy_j$ exist in $G$.
Therefore, the induced subgraph of $G$ on $\{x_i: 1\le i\le r\}\cup\{y_j: r\leq j\leq 2r\}$ has $K_{r, r}$ as an induced subgraph.
Since $r\geq n$, we know $G$ has $K_{n, n}$ as an induced subgraph.
Since the argument is symmetric for $(a,b,c,d)=(0,0,0,1)$, the only remaining case is when $(a,b,c,d)=(0,0,1,1)$.

Now, when $(a,b,c,d)=(0,0,1,1)$, we know that for $i, j\in[2r]$, both $x_iy_j$ and $y_ix_j$ exist in $G$, so $G$ has a complete bipartite graph $K_{2r, 2r}$ as an induced subgraph.
This implies that $G$ has $K_{n, n}$ as an induced subgraph, and so the theorem is proven.
\end{proof}


\section{Concluding Remarks: Generalizations}\label{SEC:inducedpacking}

Theorem~\ref{THM:inducedmatching} is easily generalized by considering the notion of an induced $\HH$-matching.
Given a set $\HH$ of connected graphs, an {\it induced $\HH$-matching} of a graph $G$ is an induced subgraph of $G$ where each component is isomorphic to a graph in $\HH$.
The {\it induced $\HH$-matching number} of $G$, denoted $\alpha _{\HH}(G)$, is the maximum number of components of an induced $\HH$-matching of $G$.
Note that $\alpha _{\{P_{1}\}}(G)$ and $\alpha _{\{P_{2}\}}(G)$ are exactly the {\it independence number} $\alpha (G)$ and the {\it induced matching number} $\alpha''(G)$, respectively, of $G$.
Moreover, if $P_{1}\in \HH$, then it is not hard to see that $\alpha_{\HH} (G)=\alpha (G)$.

Let $\HH$ be a finite set of connected graphs such that $P_{s}\in \HH$ for some $s$.
One can derive a Ramsey-type theorem for the induced $\HH$-matching number regarding connected graphs, similar to Theorem~\ref{THM:inducedmatching}.
We first provide some definitions in order to state the theorem.

Let $H$ be a connected graph and let $X$ be a non-empty subset of $V(H)$.
For a positive integer $l$, an {\it $(H, X)$-broom of length $l$} is the graph obtained from a path of length $l-1$ where one endpoint of the path is adjacent to all vertices in $X$; the other endpoint of the path is called the {\it endpoint} of the broom.

\begin{figure}[h!]
\begin{center}
\includegraphics[width=14cm]{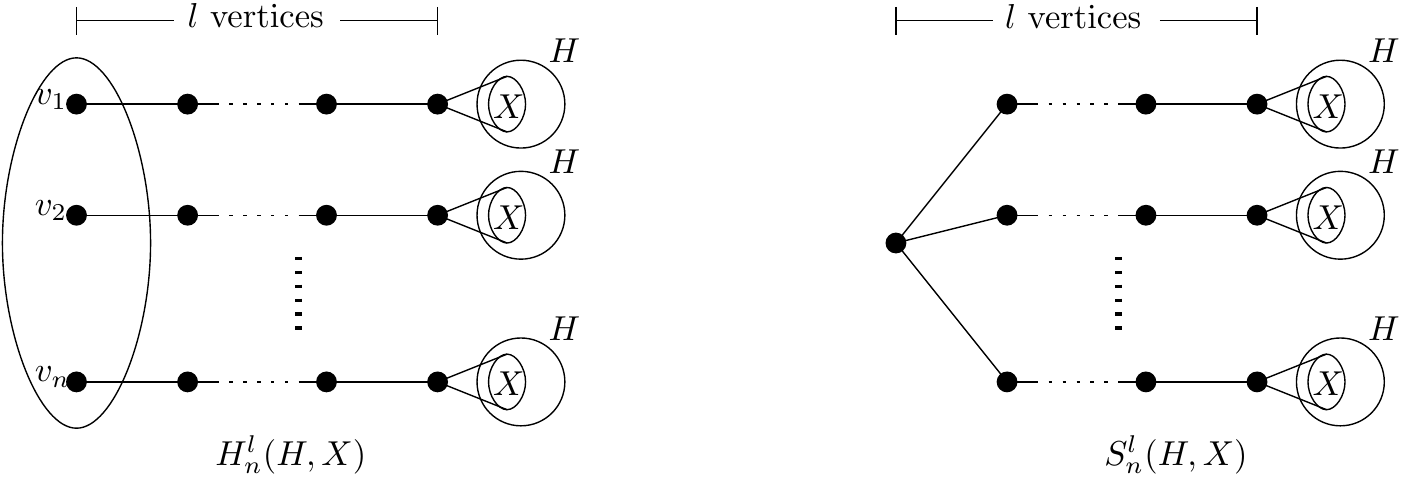}
\caption{An $(H,X)$-hairy graph and a $(H, X)$-star graph.}
\label{fig:hairy}
\end{center}
\end{figure}

We define two classes of graphs. See Figure~\ref{fig:hairy}.
In the following, $n$ and $l$ are positive integers, $H$ is a connected graph, and $X\subseteq V(H)$.
\begin{itemize}
\item \emph{$(H,X)$-hairy graph with width $n$ and length $l$, $H_n^l(H,X)$}: the graph obtained from the complete graph with $n$ vertices $v_1,\ldots,v_n$ by identifying $v_i$ and the endpoint of an $(H,X)$-broom of length $l$ for each $i$.

\item \emph{$(H,X)$-star with $n$ leaves and length $l$, $S_n^{l}(H,X)$}: the graph obtained from the disjoint union of $n$ $(H,X)$-brooms of length $l$ by identifying all $n$ endpoints of the $(H, X)$-brooms.
\end{itemize}

Generalizing our Theorem~\ref{THM:inducedmatching}, one can obtain the following Ramsey-type theorem for the induced $\HH$-matching number regarding connected graphs:

\begin{thm}\label{THM:general}
Given a positive integer $s\ge1$, let $\HH$ be a finite set of connected graphs where $P_s \in \HH$.
For every positive integer $n$, there exists an integer $N$ such that every connected graph $G$ with $\alpha_{\HH}(G)\ge N$ contains an induced subgraph isomorphic to one of the following:
\begin{itemize}
\item $P_n$,
\item $H_n^l(H,X)$ for some $l \le s$, $H\in \HH$ and $X \subseteq V(H)$, or
\item $S_n^l(H,X)$ for some $l \le s$, $H\in \HH$ and $X \subseteq V(H)$.
\end{itemize}
\end{thm}

Note that each of the guaranteed graphs also has induced $\HH$-matching number at least $n$, except $P_n$, which has induced $\HH$-matching number at least $n-s+1\over s+1$.

Theorem~\ref{THM:general} can be easily derived from the proof of Theorem~\ref{THM:inducedmatching}
by considering $G'$ (in the proof of Theorem~\ref{THM:inducedmatching}), which is the graph obtained from $G$ by 
contracting every component of a maximum induced $\HH$-matching of $G$ to a vertex, and then applying Ramsey's theorem and Lemma~\ref{LEM:increasing} repeatedly.

We remark that the case when $\HH$ does not contain a path cannot give a nice list of guaranteed graphs, since we cannot avoid all graphs obtained from attaching an arbitrary graph in $\HH$ to every pendent vertex of an arbitrary tree.


\begin{thebibliography}{1}

\bibitem{2012ALR}
A.~Atminas, V.~V. Lozin, and I.~Razgon.
\newblock Linear time algorithm for computing a small biclique in graphs
  without long induced paths.
\newblock In {\em Algorithm theory---{SWAT} 2012}, volume 7357 of {\em Lecture
  Notes in Comput. Sci.}, pages 142--152. Springer, Heidelberg, 2012.

\bibitem{2016BoPiRaSz}
C.~F. Bornstein, J.~W.~C. Pinto, D.~Rautenbach, and J.~L. Szwarcfiter.
\newblock Forbidden induced subgraphs for bounded {$p$}-intersection number.
\newblock {\em Discrete Math.}, 339(2):533--538, 2016.

\bibitem{CKO}
I.~Choi, J.~Kim, and S.~O.
\newblock The difference and ratio of the fractional matching number and the
  matching number of graphs.
\newblock {\em Discrete Math.}, 339(4):1382--1386, 2016.

\bibitem{matching}
K.~K.~Dabrowski, M.~Demange, and V.~V Lozin.
\newblock New results on maximum induced matchings in bipartite graphs and
  beyond.
\newblock {\em Theor. Comput. Sci.}, 478:33--40, 2013.

\bibitem{D}
R.~Diestel.
\newblock {\em Graph theory}, volume 173 of {\em Graduate Texts in
  Mathematics}.
\newblock Springer, Berlin, fifth edition, 2017.

\bibitem{F2018}
M.~Furuya.
\newblock {Forbidden subgraphs for constant domination number}.
\newblock {\em {Discrete Math. Theor.}},
  20:1, 2018.

\bibitem{lozin2017}
V.~V. Lozin.
\newblock Graph parameters and ramsey theory.
\newblock In {\em International Workshop on Combinatorial Algorithms}, pages
  185--194. Springer, 2017.

\bibitem{Ramsey}
F.~P. Ramsey.
\newblock On a Problem of Formal Logic.
\newblock {\em Proc. London Math. Soc. (2)}, 30(4):264--286, 1929.
\end{thebibliography}
\end{document}